\documentclass[12pt]{amsart}

\usepackage{amsmath}
  \usepackage{paralist}
  \usepackage{graphics} 
  \usepackage{epsfig} 
\usepackage{graphicx}  \usepackage{epstopdf}
\usepackage{enumitem}
 \usepackage[colorlinks=true]{hyperref}
\hypersetup{urlcolor=blue, citecolor=red}

\newtheorem*{TheoremA}{Theorem A}
\newtheorem*{TheoremB}{Theorem B}
\newtheorem*{TheoremC}{Theorem C}
\newtheorem*{TheoremD}{Theorem D}

\newtheorem*{TheoremD'}{Theorem D'}
\newtheorem*{TheoremE}{Theorem E}
\newtheorem*{TheoremE'}{Theorem E'}

\newtheorem{theorem}{Theorem}[section]

\newtheorem*{main*}{Main Theorem}
\newtheorem{lemma}[theorem]{Lemma}
\newtheorem{proposition}[theorem]{Proposition}

\theoremstyle{definition}
\newtheorem{definition}[theorem]{Definition}
\newtheorem{remark}[theorem]{Remark}

\newcommand{\CC}{\mathbb{C}}

\def\M{{\mathcal M}}

\def\RR{{\mathbb R}}
\def\NN{{\mathbb N}}

\newcommand{\fg}{\mathfrak{g}}
\newcommand{\fh}{\mathfrak{h}}

\def\top{{\mathop{\hbox{\footnotesize \rm top}}}}

\def\a{\alpha}

\def\d{\delta}   
   
 \def\e{\epsilon}

\newcommand{\dist}{\mathrm{dist}}

\newcommand{\Ad}{\mathrm{Ad}}

\newcommand{\ggm}{G/\Gamma}

\title[Running heading with forty characters or less]
      {On the continuity of topological entropy of certain partially hyperbolic diffeomorphisms}

\author[first-name1 last-name1 and first-name2 last-name2]{Weisheng Wu}


\address{School of Mathematical Sciences, Xiamen University, Xiamen, 361005, P. R. China}
\email{wuweisheng@xmu.edu.cn}

\begin{document}

\maketitle
\markboth{Continuity of topological entropy}
{W. Wu}
\renewcommand{\sectionmark}[1]{}

\begin{abstract}
In this paper, we consider certain partially hyperbolic diffeomorphisms with center of arbitrary dimension and obtain continuity properties of the topological entropy under $C^1$ perturbations. The systems considered have subexponential growth in the center direction and uniform exponential growth along the unstable foliation. Our result applies to partially hyperbolic diffeomorphisms which are Lyapunov stable in the center direction. It applies to another important class of systems which do have subexponential growth in the center direction, for which we develop a technique to use exponential mixing property of the systems to get uniform distribution of unstable manifolds. A primary example is the translations on homogenous spaces which may have center of arbitrary dimension and of polynomial orbit growth.
\end{abstract}

\section{Introduction}
Entropy is one of the most important invariants of a dynamical system which expresses the amount of ``unpredictability'' in the system. One of the central subjects in dynamical systems is to investigate the robust properties of the system under perturbations. Specifically, we are addressing in this paper the problem of the continuity properties of the topological entropy of partially hyperbolic diffeomorphisms under $C^1$ perturbations.

The concept of partial hyperbolicity was proposed independently by Brin-Pesin and Pugh-Shub in 1970s, which is a natural generalization of uniform hyperbolicity. The structural stability of uniformly hyperbolic diffeomorphisms implies the topological entropy is locally constant near such systems in the $C^1$ topology. For systems beyond uniform hyperbolicity, we first briefly survey some existing results from which we can see the key ingredients in the continuity of topological entropy.

The regularity of diffeomorphisms and the existence of homoclinic tangences play important roles in the upper semicontinuity of the topological entropy:
\begin{enumerate}
  \item By \cite{Yomdin} (also \cite{Bu, New}), the topological entropy is upper semicontinuous on the space of $C^\infty$ diffeomorphisms.
  \item Misiurewicz \cite{Mis} constructed $C^r$ counterexamples with homoclinic tangencies for any $1\le r<\infty$, at which the topological entropy is not upper semicontinuous in the $C^r$ topology.
  \item Liao, Viana and Yang \cite{LVY} proved that the topological entropy is in fact upper semicontinuous on the space
 of $C^1$ diffeomorphisms away from tangencies. In particular, the topological entropy is upper semicontinuous on the space
 of $C^1$ partially hyperbolic diffeomorphisms with one-dimensional ($1$-D for short) center.
\end{enumerate}

The lower semicontinuity of the topological entropy is usually due to the existence of certain types of dynamical structure which are robust and carry entropy:
\begin{enumerate}
  \item It is well known that horseshoes carry positive entropy and they are persistent after small perturbations. Then by a classical result of Katok \cite{Ka} the topological entropy is lower semicontinuous on the space of $C^{1+\alpha}$ diffeomorphisms on surfaces.
  \item In \cite{HSX}, the authors considered a $C^1$ partially hyperbolic diffeomorphism $f$ whose stable and unstable foliations  stably carry some unique non-trivial homologies. Then if $f$ has $1$-D center, the topological entropy is constant in a small $C^1$ neighborhood of $f$; if $f$ has $2$-D center, the topological entropy is continuous at $f$ in the $C^\infty$ topology\footnote{By \cite{HSX}, under the assumption that the unstable foliation of $f$ stably carries a unique non-trivial homology, the volume growth along unstable foliation $\chi_u(f)$ is locally constant at $f$. $\chi_u(f)$ can also be understood as the unstable topological entropy of $f$ by \cite{HHW}.}.
  \item In \cite{SY}, the authors proved for $f$ which is any small $C^1$ perturbation of a time-one map of a hyperbolic flow, the local unstable manifolds have uniform rate of expansion, which implies the (unstable) topological entropy is lower semicontinuous at $f$ in the $C^1$ topology. Since center bundle is one dimensional, the topological entropy is continuous at $f$ in the $C^1$ topology.
\end{enumerate}
Please see also the papers \cite{RSY, CLPV, BCF} for recent progress.

Let us focus on the topological entropy of $C^1$ partially hyperbolic diffeomorphisms. As one can see from above, most known results are established for partially hyperbolic diffeomorphisms with $1$-D or $2$-D center. In \cite{SY}, the authors conjectured the topological entropy is continuous on the space of $C^1$ partially hyperbolic diffeomorphisms with $1$-D center. Indeed, the conjecture is true for all the known examples. In contrast, when center has dimension two, the topological entropy can fail to be continuous; such an example is constructed in \cite{HSX}. Then it is natural to investigate the topological entropy for more partially hyperbolic diffeomorphisms with higher dimensional center.

Partial hyperbolicity includes the hyperbolic part and the center part. In \cite{HHW}, the authors introduced the notion of unstable topological entropy, which is caused by the hyperbolic part of the system. Recently, the unstable metric entropy is studied in \cite{HHW} and \cite{Yang}, see also the survey \cite{Ta}. In this paper, we focus on $C^1$ partially hyperbolic diffeomorphisms with subexponential growth in the center direction. Then only the hyperbolic part contributes to the topological entropy, which can be verified using the theory of unstable entropy. We remark that Ledrappier-Young formula which holds only for $C^r, r>1$ diffeomorphisms, is not directly applied here.

On the other hand, we investigate the continuity of unstable topological entropy in the $C^1$ topology. The upper semicontinuity turns out to be a general result for all $C^1$ partially hyperbolic diffeomorphisms. The lower semicontinuity is more delicate, and we obtain it under the condition that the unstable manifolds have uniform exponential expansion, which has appeared in \cite{SY}.

As explained above, our main result in this paper states that the topological entropy function is lower semicontinuous at partially hyperbolic diffeomorphisms which have subexponential growth in the center direction and uniform exponential growth along the unstable foliation. This result applies to partially hyperbolic diffeomorphisms which is topologically transitive and Lyapunov stable in the center direction (hence to those with bounded expansion in the center direction). This class includes time-one maps of frame flows; semisimple translations on homogenous spaces; time-one map of geodesic flows in symmetric spaces of noncompact type, etc. An interesting phenomenon is that Lyapunov stable center together with topological transitivity implies uniform exponential growth along the unstable foliation.

A more difficult situation is when the system does have subexponential growth in center direction (for example, polynomial growth). Our technique is to establish effective density of the unstable foliation under the assumption that the system has an exponential mixing property. The fast equidistribution of the unstable foliation then eliminates the effect caused by subexponential growth in the center direction, which guarantees the uniform exponential growth along the unstable foliation. An important example would be the translations on homogeneous spaces. Let $G$ be a connected semisimple Lie group, $\Gamma$ a cocompact lattice in $G$, and denote by $T_g$ the map given by the left translation on $\ggm$ by $g\in G$. When $\Ad g$ is non-quasiunipotent, the system $(\ggm,g)$ is a primary example of partially hyperbolic diffeomorphisms (cf. Theorem D in \cite{PS}). The system may have polynomial growth in the center direction, and indeed it satisfies a version of exponential mixing property (``Condition (EM)'' as in \cite[Section 2.4.1]{KM}). It is also interesting to consider $C^1$ perturbations of $T_g$ as well as the time one map of a smooth time change of a homogenous flow. We do not know if such systems have certain exponential mixing property. At last, similar results are obtained for ergodic toral automorphisms, via effective density of their unstable foliations.

\subsection{Statement of main results}
In general, for a compact metric space $(X,d)$, a continuous map $T:X \rightarrow X$ and $\forall~n\in \NN$,
we define a new distance $d_{n}$ on $X$ by
\begin{eqnarray*}
d_{n}:X~\times~X&\rightarrow&\mathbb{R}\\
(x,y) &\mapsto& d_{n}(x,y)=\max_{0\leq j<n}d(T^j(x),T^j(y)),
\end{eqnarray*}
which measures the maximal distance between the two orbits of $x$ and $y$ up to time $n$.
The new metric space $(X,d_{n})$ is obviously homeomorphic to $(X,d)$.

For any $\epsilon > 0$, we say that a subset $A \subset X$ is an $(n,\epsilon)$-separated set,
if for any distinct points $x,y\in A$,
$d_{n}(x,y)> \epsilon$.
The compactness of $X$ implies that an $(n,\e)$-separated set is a finite set.
Let
$$N(T,n,\epsilon):=\max\{\# A\ |A\ \text{ is an} ~(n,\epsilon)\text{-separated set in\ }X\}$$
where $\# A$ denotes the cardinality of $A$.
The topological entropy of the map $T$ is defined by the following formula:
\begin{equation*}
h_{\top}(T):=\lim_{\epsilon \rightarrow 0^{+}}\limsup_{n\to \infty}\frac{1}{n}\log N(T,n,\epsilon).
\end{equation*}
Topological entropy can be defined equivalently by using $(n,\e)$-spanning sets or open covers of $X$.
See Chapter $7$ in \cite{W} or Chapter $3$ in \cite{KH} for comprehensive discussions on topological entropy.

Let $M$ be a $d$-dimensional smooth, closed Riemannian manifold and let $f: M \to M$ be a $C^1$ diffeomorphism. $f$ is called a \emph{partially hyperbolic diffeomorphism} (abbreviated as \emph{PHD}) if there exists a nontrivial $Df$-invariant splitting of the tangent bundle $TM= E_f^s \oplus E_f^c \oplus E_f^u$ into so-called stable, center, and unstable distributions, such that all unit vectors $v^{\sigma} \in E_f^\sigma(x)$ ($\sigma=s, c, u$) with $x\in M$ satisfy
\begin{equation*}
\|D_xfv^s\| < \|D_xfv^c\| < \|D_xfv^u\|,
\end{equation*}
and
\begin{equation*}
\|D_xf|_{E_f^s(x)}\| <1, \ \ \ \text{\ and\ \ \ \ } \|D_xf^{-1}|_{E_f^u(x)}\| <1,
\end{equation*}
for some suitable Riemannian metric on $M$. The stable distribution $E_f^s$ and unstable distribution $E_f^u$ are integrable: there exist so-called stable and unstable foliations $W_f^s$ and $W_f^u$ respectively such that $TW_f^s=E_f^s$ and $TW_f^u=E_f^u$. We say that $f$ is \emph{dynamically coherent} if there exist $f$-invariant center-stable and center-unstable foliations $W_f^{cs}$ and $W_f^{cu}$ tangent to $E_f^{cs}$ and $E_f^{cu}$ respectively; then there exists a center foliation $W_f^c$ tangent to $E_f^c$ by intersecting leaves of $W_f^{cs}$ and $W_f^{cu}$. See for example \cite{RRU} for the basic properties of PHDs. We mention that $C^r$ PHDs form an open subset of $C^r(M)$, the space of $C^r$ diffeomorphims of $M$ for any $1\leq r\leq \infty$.

We will recall the definitions of unstable entropy $h^u_{\top}(f)$ in Section 2. Our first result below is for general PHDs.
\begin{proposition}\label{unstableusc}
Let $f: M\to M$ be any $C^1$ partially hyperbolic diffeomorphism. Then
$$h^u_{\top}: C^1(M)\to \RR^+\cup\{0\}$$
is upper semicontinuous at $f$.
\end{proposition}
Note that in a recent paper \cite{RSY}, the authors have shown that the function of unstable volume growth is upper semicontinuous in the $C^1$ topology using geometric arguments. And it is already proved in \cite{HHW} that unstable topological entropy and unstable volume growth coincide. Our proof of Proposition \ref{unstableusc} uses unstable metric entropy and its upper semicontinuity. Modulo Proposition \ref{unstableusc}, the main work in this paper is to study whether unstable topological entropy is lower semicontinuous.

We focus on a $C^1$ PHD $f:M\to M$ \emph{with subexponential growth in the center direction}, that is, for any $\e>0$, there exists a constant $C_\e>0$ such that $\|D_xf^n|_{E_f^c(x)}\|\leq C_\e e^{n\e}$ for any $n\in \mathbb{N}$ and any $x\in M$ (cf. \cite{BFH}). As for the unstable direction, we introduce the following terminology in order to simplify the presentation.

\begin{definition}\label{ueg}
We say that a $C^1$ PHD $f: M\to M$ has \emph{uniform exponential growth along the unstable foliation} if for any $\rho>0$ and $\d>0$ small enough, there exists $N(\rho,\d)\in \NN$ such that for any $x\in M$, $f^{N(\rho, \d)}W^u_f(x,\d)$ contains at least $e^{N(\rho, \d)(h-3\rho)}$ disjoint sets of the form $\{W^u_f(z_i,2\d): z_i\in M\}$, where $h=h_\top^u(f)$ and $W^u_f(x,\d)$ denotes the open ball centered at $x$ of radius $\d$ with respect to the metric $d^u$ inherited from the Riemannian structure on $W^u(x)$.
\end{definition}

Our main result in this paper is the following:
\begin{TheoremA}
Let $f: M\to M$ be a $C^1$ partially hyperbolic diffeomorphism.
\begin{enumerate}
  \item If $f$ has uniform exponential growth along the unstable foliation, then the unstable topological entropy function $h^u_{\top}: C^1(M)\to \RR^+\cup\{0\}$ is continuous at $f$.
  \item Assume further that $f$ has subexponential growth in the center direction, then the topological entropy function $h_{\top}: C^1(M)\to \RR^+\cup\{0\}$ is lower semicontinuous at $f$. If in addition $f$ is $C^\infty$, $h_{\top}: C^\infty(M)\to \RR^+\cup\{0\}$ is continuous at $f$.
\end{enumerate}
\end{TheoremA}

We apply Theorem A to several important classes of PHDs.

A PHD is called \emph{Lyapunov stable in the center direction (or just Lyapunov stable)} if given
any $\e> 0$, one can find $\d> 0$ such that for any $C^1$-arc $\gamma$ tangent to $E^c$ with length
less than $\d$, the curves $f^n\gamma, n\in \mathbb{N}$ all have length less than $\e$ (cf. \cite{RRU0}). When $f$ and $f^{-1}$ are both Lyapunov stable, we say that $f$ is bi-Lyapunov stable.
It is known that if $f$ is Lyapunov stable then $E^c\oplus E^s$ integrates to a unique $f$-invariant foliation $W^{cs}$. In particular, if $f$ is bi-Lyapunov stable, then $f$ is dynamically coherent (cf. Theorem 7.5 in \cite{RRU0}).

A $C^1$ PHD $f:M\to M$ is said to have \emph{bounded expansion in the center direction} if there exists $K>1$ such that
$$\|Df^n(x)|_{E_f^c(x)}\|< K, \quad \forall n\in \mathbb{N}.$$
It is clear that bounded expansion implies both Lyapunov stability and subexponential growth in the center direction.

\begin{TheoremB}
Let $f: M \to M$ be a topologically transitive $C^1$ partially hyperbolic diffeomorphism which is Lyapunov stable in the center direction. Then $h^u_{\top}: C^1(M)\to \RR^+\cup\{0\}$ is continuous at $f$. If in addition $f$ has bounded expansion in the center direction, then the topological entropy function $h_{\top}: C^1(M)\to \RR^+\cup\{0\}$ is lower semicontinuous at $f$.
\end{TheoremB}

Our technique to prove Theorem B, particularly the uniform exponential growth of $f$ along the unstable foliation, is to use effective density of the unstable foliation of $f$. The technique applies to another important class of PHDs which we describe below.

Consider a volume preserving PHD $f:M\to M$, i.e., $f$ preserves a smooth probability measure $m$ on $M$. Denote by $m_x^u$ the volume induced by the Riemannian structure on $W^u(x)$. For any $r\in \NN\cup \{\infty\}$, let $C^r(M,\RR)$ denote the set of all the $C^r$ functions from $M$ to $\RR$ equipped with the $C^r$-norm $\|\cdot\|_r$.
We say $f$ has \emph{property (EM)}, if there exists $r\in \NN$, $C>0$, $\a>0$ and $\d_0>0$, such that for any $0<\d\leq \d_0$, any $\phi, \psi\in C^r(M,\RR)$,
\begin{equation}\label{e:mixing0}
\begin{aligned}
\left|\int_{W_f^u(x,\d)}\phi(f^n(p))\psi(p)dm_x^u(p)-\int\phi dm\int_{W_f^u(x,\d)}\psi dm_x^u\right|<C\|\phi\|_r\|\psi\|_re^{-n\a}
\end{aligned}
\end{equation}
for any $n\in \mathbb{N}$ and any $x\in M$. See \cite{DHK} for a related notion of exponential mixing.

\begin{TheoremC}
Let $f: M \to M$ be a $C^1$ dynamically coherent partially hyperbolic diffeomorphism which has subexponential growth in the center direction and satisfies property (EM). Then the topological entropy function $h_{\top}: C^1(M)\to \RR^+\cup\{0\}$ is lower semicontinuous at $f$. Furthermore, $h^u_{\top}: C^1(M)\to \RR^+\cup\{0\}$ is continuous at $f$.
\end{TheoremC}

An important example of PHDs with property (EM) is the partially hyperbolic translations on compact homogeneous spaces. Let $G$ be a connected semisimple Lie group and $\Gamma\subset G$ be a cocompact lattice. Choose and fix a norm on the Lie algebra $\fg$ of $G$, which induces a right invariant Riemannian metric ``dist" on $G$. Then set $$d(g\Gamma, h\Gamma):=\inf_{\gamma\in \Gamma}\dist(g,h\gamma).$$
We study the topological entropy of the dynamical system $T_g: \ggm\to \ggm$ induced by the left translation of $T_g(x)=gx$, for any $x\in \ggm$. Note that $T_g:\ggm\to \ggm$ is a $C^\infty$ PHD when $\Ad g$ is nonquasiunipotent, see Section $5.2$ for more details. The polynomial orbit growth in the center direction causes substantial difficulties in many important problems in homogenous dynamics (see for example \cite{KM, Kan}). The following theorem is a corollary of Theorem C.
\begin{TheoremD}
Let $G$ be a connected semisimple Lie group, $\Gamma$ be a cocompact lattice of $G$, and $g\in G$. Then the topological entropy function $h_{\top}: C^1(\ggm)\to \RR^+\cup\{0\}$ is lower semicontinuous at $T_g$ and $h_{\top}: C^\infty(\ggm)\to \RR^+\cup\{0\}$ is continuous at $T_g$.
\end{TheoremD}

\begin{remark}
The term ``Condition (EM)'' is used in \cite[Section 2.4.1]{KM} for $(\ggm, g^t)$ where $g^t$ is a one-parameter subgroup of $G$ (other conditions on $\ggm$ are imposed), from which our property (EM) \eqref{e:mixing0} can be derived (cf. \cite[Proposition 2.4.8]{KM}). In our discrete setting, property (EM) follows from \cite[Propositions 3.1]{GSW} (see Section 5.2). It is interesting to know whether the time-one map of a smooth time change of $(\ggm, g^t)$ satisfies property (EM).
\end{remark}

\begin{TheoremE}
Let $T_A: \mathbb{T}^d\to \mathbb{T}^d$ be an ergodic toral automorphism. Then the topological entropy function $h_{\top}: C^1(\mathbb{T}^d)\to \RR^+\cup\{0\}$ is lower semicontinuous at $T_A$. Furthermore, $h^u_{\top}: C^1(\mathbb{T}^d)\to \RR^+\cup\{0\}$ is continuous at $T_A$, and $h_{\top}: C^\infty(\mathbb{T}^d)\to \RR^+\cup\{0\}$ is continuous at $T_A$.
\end{TheoremE}

\begin{remark}
We will state Theorem D', an improvement of Theorem D in Section 5.2, and give a simpler proof. Similarly, Theorem E can also be improved, see Theorem E' in Section 5.3. Nevertheless, our technique provides a systematic way to treat these problems.
\end{remark}

The paper is organized as follows. In Section $2$, we recall the notion of unstable entropy and prove Proposition \ref{unstableusc}. Section 3 is devoted to proving main Theorem A. In section 4, we consider PHDs with Lyapunov stable center and prove Theorem B. Theorems C, D and E are proved in the last section, where we also summarize the known examples of PHDs at which the topological entropy has continuity properties.

\section{Unstable topological entropy}
In this section, we collect some results on unstable topological entropy, and prove Proposition \ref{unstableusc}. The results in this section except Proposition \ref{unstable} hold for general PHDs.

Let $M$ be a $d$-dimensional smooth, closed manifold and let $f: M \to M$ be a $C^1$ PHD. We start to recall the definition of unstable topological entropy given in \cite{HHW}.

Denote by $d^\sigma$ the metric induced by the Riemannian structure
on $W^\sigma_f$, $\sigma\in \{s,u,cs,cu,c\}$ if the foliation exists, and then denote by $W_f^\sigma(x,\delta)$ the open ball inside $W_f^\sigma(x)$ centered at $x\in M$
of radius $\delta>0$ with respect to the metric $d^\sigma$.
Given any $y\in W_f^u(x)$, define
$$d^u_{n}(x,y):=\max _{0 \leq j <n}d^u(f^j(x),f^j(y)).$$
Let $N^u(f,n,\epsilon,x,\delta)$ be the maximal
number of points in $\overline{W_f^u(x,\delta)}$ with pairwise
$d^u_{n}$-distances at least $\epsilon$.  We call such set a maximal
\emph{$(n,\epsilon)$ u-separated set} of $\overline{W_f^u(x,\delta)}$.

\begin{definition}\label{Defutopent1}
The \emph{unstable topological entropy} of $f$ on $M$ is defined by
\begin{equation*}
h^u_{\text{top}}(f)
:=\lim_{\delta \to 0}\sup_{x\in M}h_{\text{top}}(f, \overline{W_f^u(x,\delta)}),
\end{equation*}
where
\begin{equation*}
 h_{\text{top}}(f, \overline{W_f^u(x,\delta)})
:=\lim_{\epsilon \to 0}\limsup_{n\to \infty}\frac{1}{n}\log N^u(f,n,\epsilon,x,\delta).
\end{equation*}
\end{definition}
\begin{remark}\label{unstableentropy}
One can also use $(n,\epsilon)$ u-spanning sets or open covers to give equivalent
definitions for unstable topological entropy. Moreover,
$$h^u_{\text{top}}(f)
=\sup_{x\in M}h_{\text{top}}(f, \overline{W_f^u(x,\delta)})$$
for any $\d>0$ small enough (cf. Lemma 4.1 in \cite{HHW}).
\end{remark}

The unstable volume growth (cf. \cite{HSX}) is defined as follows:
\begin{equation*}
\chi_u(f):= \sup_{x \in M} \chi_u(x, \delta)
\end{equation*}
where
\begin{equation*}
\chi_u(x, \delta) := \limsup_{n \to \infty} \frac{1}{n} \log
(\text{Vol} (f^n(W_f^u(x,\delta))).
\end{equation*}
It is independent of the choice of $\delta>0$ and the Riemannian metric. As mentioned in the introduction, by Theorem C in \cite{HHW},
$$h^u_{\text{top}}(f)=\chi_u(f).$$

Let $\mu$ be an $f$-invariant Borel probability measure on $M$. In \cite{HHW}, the authors give a new definition for the unstable metric entropy $h_{\mu}^u(f)$ which coincides with the one considered by Ledrappier and Young \cite{LY2}, by using measurable partitions subordinate to unstable manifolds that can be obtained by refining a finite Borel partition into pieces of unstable leaves. We skip this definition here, just recalling two important properties of unstable metric entropy as follows.

Let $\mathcal{M}_f(M)$ (resp. $\mathcal{M}^e_f(M)$) denote
the set of all $f$-invariant (resp. $f$-ergodic) Borel probability measures on $M$.
The following variational principle is proved in Theorem D of \cite{HHW}.
\begin{theorem} \label{vp}
Let $f: M\to M$ be any $C^1$ PHD. Then
$$h^u_{\top}(f)=\max_{\mu \in \mathcal{M}_f(M)}h_{\mu}^u(f)=\max_{\mu \in \mathcal{M}^e_f(M)}h_{\mu}^u(f).$$
\end{theorem}

The upper semicontinuity of the unstable metric entropy function
$$h^u_{\{\cdot\}}(f): \M_f(M)\to \RR^+\cup\{0\}$$
is proved in Proposition 2.15 in \cite{HHW}. In fact, Theorem A in \cite{Yang} gives the following more general result:
\begin{theorem} \label{usc}
Let $f_n, f:M\to M$ be $C^1$ PHDs such that $f_n\to f$ in the $C^1$ topology as $n\to \infty$. Suppose that $\mu_n\in \M_{f_n}(M)$, and $\mu_n\to \mu$ in the weak$^*$ topology as $n\to \infty$. Then
$$\limsup_{n\to \infty}h^u_{\mu_n}(f_n)\leq h^u_\mu(f).$$
\end{theorem}

We are ready to prove Proposition \ref{unstableusc}, that is, $h^u_{\top}: C^1(M)\to \RR^+\cup\{0\}$
is upper semicontinuous at any $C^1$ PHD $f$.
\begin{proof}[Proof of Proposition \ref{unstableusc}]
For any $C^1$ PHD $f: M\to M$, let $f_n\to f$ in the $C^1$ topology as $n\to \infty$. Then $f_n$ are also PHDs. According to Theorem \ref{vp}, we know that for each $f_n$, there exists $\mu_n\in \M_{f_n}(M)$ such that
\begin{equation}\label{e:max}
h^u_{\top}(f_n)=h^u_{\mu_n}(f_n).
\end{equation}
Let $n_k$ be a subsequence such that
\begin{equation}\label{e:limit}
\limsup_{n\to \infty}h^u_{\top}(f_n)=\lim_{k\to \infty}h^u_{\top}(f_{n_k}).
\end{equation}
Then a subsequence of $\mu_{n_k}$ converges to a probability measure $\mu$ on $M$. Without loss of generality, we assume $\mu_{n_k}\to \mu$.
It is straightforward to verify that $\mu \in \M_{f}(M)$. By Theorem \ref{usc}, we have
\begin{equation}\label{e:less}
\lim_{k\to \infty}h^u_{\mu_{n_k}}(f_{n_k})\leq h^u_\mu(f).
\end{equation}
By Theorem \ref{vp}, $h^u_\mu(f)\le h^u_\top(f)$. Thus by \eqref{e:max},\eqref{e:limit} and \eqref{e:less},
$$\limsup_{n\to \infty}h^u_{\top}(f_n)\leq h^u_\top(f).$$
The proof of the proposition is complete.
\end{proof}

Let us consider a $C^1$ PHD $f:M\to M$ with subexponential growth in the center direction, that is, for any $\e>0$, there exists a constant $C_\e>0$ such that $\|D_xf^n|_{E_f^c(x)}\|\leq C_\e e^{n\e}$ for any $n\in \mathbb{N}$ and any $x\in M$. The following lemma is useful for $C^1$ PHDs.

\begin{lemma}\label{unstablevolume}(Cf. \cite{HSX})
Let $f:M\to M$ be a $C^{1}$ PHD and $\mu$ an $f$-invariant and ergodic measure. Then
$$h_\mu(f)\leq \chi^u(f)+\sum_{\lambda_i^c(\mu)>0}m_i(\mu)\lambda_i^c(\mu)$$
where the sum is taken over all the positive Lypapunov exponents $\lambda_i^c(\mu)$ of $\mu$ in the center direction with multiplicity $m_i(\mu)$.
\end{lemma}

\begin{proposition}\label{unstable}
If a $C^{1}$ PHD $f:M\to M$  has subexponential growth in the center direction, then
$h_{\top}(f)=h^u_{\top}(f)$.
\end{proposition}
\begin{proof}
Let $\nu$ be any $f$-invariant and ergodic Borel probability measure on $M$. It is clear that the Lyapunov exponents of $\nu$ in the center direction are always nonpositive, since $f$ has subexponential growth in the center direction. Thus by Lemma \ref{unstablevolume} and the fact $h^u_{\text{top}}(f)=\chi_u(f)$, we have $h_\nu(f)\le h^u_{\text{top}}(f)$. Thus by the variational principle, we have
$$h_{\top}(f)=\sup_{\nu \in \mathcal{M}^e_{f}(M)}h_{\nu}(f)\le h^u_{\top}(f).$$
$h^u_{\top}(f)\leq h_{\top}(f)$ is automatic and therefore we finish the proof of the proposition.
\end{proof}

\begin{remark}
If $f:M\to M$ is a $C^r$-PHD, $r>1$, we can also obtain $h_\nu(f)=h^u_\nu(f)$ and $h^u_{\top}(f)= h_{\top}(f)$, by using Ledrappier-Young formula (cf. \cite{LY2}).
\end{remark}

\begin{remark}
A $C^1$ PHD $f:M\to M$ has subexponential growth in the center direction if and only if the Lyapunov exponents in the center direction are all nonpositive with respect to any $f$-invariant measure. The ``only if'' part is clear. For the ``if'' part, assume that for some $\e>0$, there exists a sequence of points $x_n\in M$ such that
$\|D_{x_n}f^n|_{E_f^c(x_n)}\|\ge e^{n\e}$
for $n\ge 0$. Then a subsequence of the measures
$$\mu_n:=\frac{1}{n}\sum_{i=0}^nf^i_*\delta_{x_n}$$
converges to an $f$-invariant measure $\mu$, which has some positive central Lyapunov exponents, so we get a contradiction.
\end{remark}

\section{Proof of Theorem A}
In this section we prove Theorem A. Recall that a $C^1$ PHD $f: M\to M$ has uniform exponential growth along the unstable foliation, if for any $\rho>0$ and any $\d>0$ small enough, there exists $N(\rho,\d)\in \NN$ such that for any $x\in M$, $f^{N(\rho, \d)}W^u_f(x,\d)$ contains at least $e^{N(\rho, \d)(h-3\rho)}$ disjoint sets of the form $\{W^u_f(z_i,2\d): z_i\in M\}$, where $h=h_\top^u(f)$.

Let $D^k$ denote the unit open disk in $\RR^k$ and $d=\dim M$. Let $\mathcal F$ be a continuous foliation of $M$ with $C^1$ leaves, then an \emph{$\mathcal F$-foliation box} is the image $B$ of a topological embedding $\Phi: D^{d-l}\times D^l\to M$
which sends each plaque $\Phi(\{x\}\times D^l)$ to a leaf of $\mathcal F$. Each $\Phi(x,\cdot): D^l\to M, y\mapsto \Phi(x,y)$ is a $C^1$ embedding depending continuously on $x$ in the $C^1$ topology. We write this $\mathcal F$-foliation box by $\{B, \Phi, D\}$ where $D=\Phi(D^{d-l}\times \{0\})$.

\begin{proposition}\label{ggrowth}
Let $f: M\to M$ be a $C^1$ PHD with uniform exponential growth along the unstable foliation. Then for any given $\rho>0$, there exists a $C^1$ open neighborhood $\mathcal{U}_{f}(\rho)$ of $f$ such that for any $g\in \mathcal{U}_{f}(\rho)$ and any $x\in M$,  $g^{N(\rho, \d)}W^u_g(x,\d)$ contains at least $e^{N(\rho, \d)(h-3\rho)}$ disjoint sets of the form $W^u_g(z_i,\d)$ with $z_i\in M$.
\end{proposition}
\begin{proof}
If $f_n \to f$ in the $C^1$ topology, then the foliations $W^u_{f_n}$ converge to $W_f^u$ in the following sense:
\begin{enumerate}
  \item For each $n$, there exist finite covers of $M$ by $W^u_{f_n}$- and $W^u_f$-foliation boxes
$$\{B^n_i, \Phi^n_i, D_i\}^k_{i=1} \quad \text{and\quad }\{B_i, \Phi_i, D_i\}^k_{i=1}$$
respectively;
  \item for each $1\le i\le k$, the topological embeddings
  $$\Phi^n_i: D^{d-l}\times D^l\to M$$
converge uniformly to $\Phi_i$ in the $C^0$ topology as $n\to \infty$;
  \item for every $1\le i\le k$ and $x\in D_i, \Phi_i^n(x, \cdot): D^l\to M$ defined by
$$y\mapsto \Phi_i^n(x,y)$$
are $C^1$ embeddings converging to $\Phi_i(x, \cdot)$ in the $C^1$ topology as $n\to \infty$.
\end{enumerate}

Now $f$ has uniform exponential growth along the unstable foliation. Without loss of generality, one can choose $\d>0$ to be much smaller than the Lebesgue number of the open cover $\{B_i\}_{i=1}^k$ of $M$. Since $N(\rho, \d)$ is already fixed, the conclusion of the proposition follows from the above continuity of unstable foliations and the uniform exponential growth of $f$ along the unstable foliation.
\end{proof}

\begin{proof}[Proof of Theorem A]
Assume that $f$ has uniform exponential growth along the unstable foliation. By Proposition \ref{ggrowth}, for any $\rho>0$, there exists a $C^1$ open neighborhood $\mathcal{U}_{f}(\rho)$ of $f$ such that for any $g\in \mathcal{U}_{f}(\rho)$, and for any $x\in M$,  $g^{N(\rho, \d)}W^u_g(x,\d)$ contains at least $e^{N(\rho, \d)(h-3\rho)}$ disjoint sets of the form $W^u_g(z_i,\d)$ with $z_i\in M.$ By induction, for any $l\in \NN$, $g^{lN(\rho, \d)}W^u_g(x,\d)$ contains at least $e^{lN(\rho, \d)(h-3\rho)}$ disjoint sets of the form $W^u_g(w_i,\d)$ with $w_i\in M.$ Then $\{g^{-lN(\rho, \d)}w_i\}$ is a $(lN(\rho, \d), \d)$ u-separated subset of $W^u_g(x,\d)$. So
$$N^u(g,lN(\rho, \d),\d, x_0,\d)\ge e^{lN(\rho, \d)(h-3\rho)},$$
which implies
\begin{equation}\label{e:lower}
h_{\top}(g)\ge h^u_{\top}(g)\ge h-3\rho.
\end{equation}
This gives the lower semicontinuity in $C^1$ topology of the function $h^u_{\top}(\cdot)$ at $f$.
The upper semicontinuity follows from Proposition \ref{unstableusc} and hence $h^u_{\top}: C^1(M)\to \RR^+\cup\{0\}$ is continuous at $f$. This proves Theorem A(1).

If $f$ also has subexponential growth in the center direction, then $h_{\top}(f)=h^u_{\top}(f)$ by Proposition \ref{unstable}. So the topological entropy function $h_{\top}: C^1(M)\to \RR^+\cup\{0\}$ is lower semicontinuous at $f$ by \eqref{e:lower}.
If in addition $f$ is $C^\infty$, the upper semicontinuity of $h_{\top}:C^\infty(M)\to \RR^+\cup\{0\}$ is a classical result (cf. for example \cite{New}). Therefore, $h_{\top}: C^\infty(M)\to \RR^+\cup\{0\}$ is continuous at $f$. The proof of Theorem A is complete.
\end{proof}

\section{Lyapunov stable PHDs}
Recall that a PHD is called Lyapunov stable (in center direction) if given
any $\e> 0$, one can find $\d> 0$ such that for any $C^1$-arc $\gamma$ tangent to $E^c$ with length
less than $\d$, the curves $f^n\gamma, n\in \mathbb{N}$ all have length less than $\e$.
If $f$ is Lyapunov stable then $E^c\oplus E^s$ integrates to a unique $f$-invariant foliation $W^{cs}$ (cf. \cite{RRU0}).

A $C^1$ PHD $f:M\to M$ is said to have bounded expansion in the center direction if there exists $K>1$ such that
$$\|Df^n(x)|_{E_f^c(x)}\|< K, \quad \forall n\in \mathbb{N}.$$
Bounded expansion implies both Lyapunov stability and subexponential growth in the center direction.

\subsection{Proof of Theorem B}
We will show that a topologically transitive and Lyapunov stable $C^1$ PHD $f:M\to M$ has uniform exponential growth along unstable manifolds, and hence Theorem A(1) applies. If further, $f$ has bounded expansion in the center direction, then $f$ has subexponential growth in center direction and Theorem A(2) applies.

The following property is a version of effective density of the unstable foliation. 
\begin{lemma}\label{uniformtran}(Effective density, cf. \cite[Lemma 6.2]{CPZ})
Let $f:M\to M$ be a $C^1$ PHD which is topologically transitive and Lyapunov stable. Then given any $\d>0$, there is $n\in \NN$ such that for every $x, y\in M$, there is $0\leq k\leq n$ such that
$$f^k(W_f^u(x,\d))\cap W_f^{cs}(y,\d)\neq \emptyset.$$
\end{lemma}

\begin{proposition}\label{ueg0}
Let $f:M\to M$ be a $C^1$ PHD which is topologically transitive and Lyapunov stable. Then $f$ has uniform exponential growth along unstable foliation.
\end{proposition}
\begin{proof}
Let $h=h_\top^u(f)$ and $\d_0>0$ be small enough to be specified later. Recall the definition of unstable topological entropy and Remark \ref{unstableentropy}.
For any $\rho>0$, we can find $x_0\in M$ and $0<\d\ll \d_0$ small enough such that
$$h_{\top}(f, \overline{W^u_f(x_0,\d)})>h-\rho,$$
and if $d^{cs}(x,y)< \d$ then $d^{cs}(f^nx,f^ny)\ll \d_0/2$ by Lyapunov stability of $f$.
Let $0<\e\ll \d$ be small enough and $N_1\in \NN$ large enough such that if $n\ge N_1$, then
there exists an $(n,\e)$ u-separated subset $S$ of $\overline{W^u_f(x_0,\d)}$ such that $\#S=N^u(f,n,\e, x_0,\d)>e^{n(h-2\rho)}$.

Since $f$ is Lyapunov stable, the center-stable foliation $W_f^{cs}$ exists and it is continuous. Shrinking $\d_0$ if necessary, for any $y\in W^{cs}_f(x,\d_0)$ and any $0<\d<\d_0$, the $cs$-holonomy map $H_{x,y}^{cs}$ is well defined, and such that
$$H^{cs}_{x,y}(\overline{W^u_f(x,\d)})\subset W^u_f(y,c_2(\d))$$
and
\begin{equation*}
H^{cs}_{x,y}(W^u_f(x,\e/2))\supset W^u_f(y,c_3(\e,\d))
\end{equation*}
for some constants $c_2=c_2(\d)>2\d$ and $c_3=c_3(\e,\d)>0$.
Then take $N_3\in \NN$ large enough such that for any $n\ge N_3$ and any $x\in M$:
\begin{enumerate}
  \item $f^n(W^u_f(x,c_3))\supset W^u_f(f^n(x),2\d)$;
  \item $f^n(W^u_f(x,\d))\supset W^u_f(f^n(x),2c_2)$.
\end{enumerate}

Now take any $x\in M$. Applying Lemma \ref{uniformtran} with $x=f^{N_3}x$ and $y=x_0$, there exists $N_2\in \NN$ depending only on $\d$ such that
$$f^k(W_f^u(f^{N_3}x,\d))\cap W_f^{cs}(x_0,\d)\neq \emptyset$$
for some $0\leq k\le N_2$. Then $f^{k+N_3}(W_f^u(x,\d))\cap W_f^{cs}(x_0,\d)\neq \emptyset$. Let
\begin{equation*}\label{e:intersection0}
y\in f^{k+N_3}(W_f^u(x,\d))\cap W_f^{cs}(x_0,\d),
\end{equation*}
then we have $f^{k+N_3}(W_f^u(x,\d))\supset W^u_f(y,c_2)$ by $c_2>2\d$ and the choice of $N_3$. So
$$H_{x_0,y}^{cs}(\overline{W^u_f(x_0,\d)})\subset f^{k+N_3}(W_f^u(x,\d)).$$
Note that $k$ is depending on $x,x_0$, but we shall use $N_2$ instead in the definition of $N(\rho,\d)$ later.

Finally, we take $n_\rho\in \NN$ large enough such that $n_\rho>\max\{N_1, N_2, N_3\}$ and
$$\frac{n_\rho}{2N_3+N_2+n_\rho}>\frac{h-3\rho}{h-2\rho}.$$

The cs-holonomy map
$$H^{cs}_{x_0,y}: \overline{W^u_f(x_0,\d)}\to W^u_f(y,c_2)\subset f^{k+N_3}(W_f^u(x,\d))$$
is well defined by the choice of $N_3$ and $c_2$.
Let $S$ be an $(n_\rho,\e)$ u-separated subset of $\overline{W^u_f(x_0,\d)}$ of maximal cardinality. Denote $S'=H^{cs}_{x_0,y}(S)$.
Since $W^u$ and $W^{cs}$ are uniformly transverse foliations, by shrinking $\d_0>0$ if necessary, there exists some constant $c_0>1$ such that
$$d_f^{cs}(x_i,y_i)\le c_0\d$$
for any $x_i\in S$ and $y_i=H^{cs}_{x_0,y}(x_i)\in S'$.

Since $S$ is $(n_\rho,\e)$ u-separated, we know that
$$\{W_f^u(f^{n_\rho}x_i, \e/2): x_i\in S\}$$
are disjoint balls contained in $f^{n_\rho}W^u_f(x_0,\d)$.
Recall that by Lyapunov stability of $f$, $0<\d\ll \d_0$ is chosen such that
$$d^{cs}(f^{n_\rho}x_i,f^{n_\rho}y_i)< \d_0/2.$$
Then by the definition of $c_3=c_3(\e,\d)$ we have
$$H^{cs}_{f^{n_\rho}x_i,f^{n_\rho}y_i}(W_f^u(f^{n_\rho}x_i, \e/2))\supset W^u(f^{n_\rho}y_i, c_3).$$
Moreover, $\{W_f^u(f^{n_\rho}y_i, c_3): y_i\in S'\}$
are disjoint sets contained in $f^{n_\rho}W^u_f(y,c_2)$.

By the choice of $N_3$,
$$f^{N_3}(W^u_f(f^{n_\rho}y_i,c_3)\supset W^u_f(f^{N_3+n_\rho}y_i,2\d).$$
It follows that
$$\{W^u_f(f^{N_3+n_\rho}y_i,2\d): y_i\in S'\}$$
are disjoint sets contained in $f^{N_3+n_\rho}W^u_f(y,c_2)\subset f^{n_\rho+k+2N_3}W^u_f(x,\d)$.
 Since $\#S'=\#S>e^{n_\rho(h-2\rho)}$, we know
that $f^{n_\rho+k+2N_3}W^u_f(x,\d)$ contains at least $e^{n_\rho(h-2\rho)}$ disjoint sets of the form $\{W^u_f(z_i,2\d): z_i=f^{N_3+n_\rho}y_i\in M\}.$
Then it is easy to see that $f^{N(\rho,\d)}W^u_f(x,\d)$ where $N(\rho,\d)=n_\rho+N_2+2N_3$ contains at least $e^{n_\rho(h-2\rho)}>e^{N(\rho,\d)(h-3\rho)}$ disjoint sets of the form
$$\{W^u_f(w_i,2\d): w_i\in M\}.$$
Thus $f$ has uniform exponential growth along unstable foliation.
\end{proof}

\begin{proof}[Proof of Theorem B]
Theorem B is an immediate consequence of Proposition \ref{ueg0} and Theorem A.
\end{proof}

\subsection{Examples}
\subsubsection{Time-one maps of frame flows}
Let $M$ be a closed oriented $n$-dimensional manifold of negative sectional curvature.
Consider the unit tangent bundle $SM=\{(x,v): x\in M, v\in T_xM, \|v\|=1\}$, and the frame bundle
\begin{equation*}
\begin{aligned}
FM=\{&(x, v_0, v_1, \cdots, v_{n-1}): x\in M, v_i\in T_xM, \\
&\{v_0, \cdots, v_{n-1}\} \text{\ is a positively oriented orthonormal frame at\ } x\}.
\end{aligned}
\end{equation*}
The geodesic flow $g^t:SM\to SM$ is defined as
$$g^t(x,v)=(\gamma_{(x,v)}(t),\dot{\gamma}_{(x,v)}(t))$$
where $\gamma_{(x,v)}(t)$ is the unique geodesic determined by the vector $(x,v)$. The frame flow $f^t:FM\to FM$ is defined by
$$f^t(x, v_0, v_1, \cdots, v_{n-1})= (g^t(x,v_0), \Gamma_\gamma^t(v_1),\cdots, \Gamma_\gamma^t(v_{n-1}))$$
where $\Gamma_\gamma^t$ is the parallel transport along the geodesic $\gamma_{(x,v_0)}(t)$.
There is a natural fiber bundle $\pi: FM\to SM$ that takes a frame to its first vector, and then each fiber can be identified with $SO(n-1)$. We have $\pi \circ f^t=g^t\circ \pi$ and $f^t$ acts isometrically along the fibers.

Let $f=f^1$ be the time-one map of the frame flow $f^t$. Then $f$ is a PHD with splitting $TFM = E_f^s \oplus E_f^c \oplus E_f^u$, where $\dim E_f^c = 1+\dim SO(n-1)$ and $f$ acts isometrically on the center bundle. If in addition $f$ is topologically transitive, then Theorem B applies to $f$. The following are the cases when $f$ is known to be ergodic with respect to the normalized volume measure on $FM$, and hence topologically transitive.
\begin{proposition}(Cf. Theorem 0.2 in \cite{BP})
Let $f^t$ be the frame flow on an $n$-dimensional compact
smooth Riemannian manifold with sectional curvature between $-\Lambda^2$ and $-\lambda^2$ for some
$\Lambda, \lambda>0$. Then in each of the following cases the flow and its time-one map are ergodic:
\begin{itemize}
  \item if the curvature is constant,
  \item for a set of metrics of negative curvature which is open and dense in the $C^3$ topology,
  \item if $n$ is odd and $n\neq 7$,
  \item if $n$ is even, $n\neq 8$, and $\lambda/\Lambda>0.93$,
  \item if $n=7$ or $8$ and $\lambda/\Lambda>0.99023...$.
\end{itemize}
\end{proposition}

\subsubsection{Semisimple translations on homogeneous spaces}
In this example, we consider a special class of left translations on homogenous spaces. Let $G$ be a connected semisimple Lie group, $\Gamma$ an irreducible lattice in $G$, and $g\in G$. We say that $g$ is semisimple, if $\Ad g|_{\fh^0_\CC}$ is diagonalizable over $\CC$, where $\fh^0_\CC:=\bigoplus_{|\lambda|= 1} E_\lambda$ (see \eqref{e:spaces} in Section 5.2). If $\Ad g$ is non-quasiunipotent, then $T_g$ is a $C^\infty$ PHD with bounded expansion in center direction. In fact, $T_g$ acts isometrically in the center direction. Moreover, $T_g$ is ergodic (in fact, Kolmogorov by \cite{Dan1,Dan2} and Bernoulli by \cite{Kan}) with respect to the Haar measure on $\ggm$, hence it is topologically transitive. Thus Theorem B applies to $T_g$.

\subsubsection{Time-one maps of geodesic flows on symmetric spaces of noncompact type}
If the symmetric space of noncompact type has rank one, then the geodesic flow is an Anosov flow. In this case, the topological entropy is continuous at the time-one maps of geodesic flows (and their perturbations) in the $C^1$ topology (\cite{Hua, HZ, SY}).

Now we consider symmetric spaces of noncompact type with rank $\ge 2$. Let $(X, g)$ be a simply connected Riemannian symmetric space, then $X=G/K$, where $G$ is the connected component of the isometry group of $X$ which acts transitively on $X$, $K$ is the isotropy group of any point which is compact. $X$ is called a symmetric space of noncompact type if $G$ is semisimple with no compact factors and $K$ is a maximal compact subgroup of $G$. Assume $\text{rank}(X)\ge 2$. Let $\Gamma \subset G$ be a cocompact
lattice, then $M = \Gamma\setminus X$ is a compact Riemannian manifold. Consider the geodesic flow $g^t:SM\to SM$ on the unit tangent bundle of $M$
and its time-one map $g=g^1$.  Then $g$ is a $C^\infty$ PHD with higher dimensional center. We note that the center direction has polynomial orbit growth and $g^t$ is not ergodic with respect to the Liouville measure $m$ on $SM$.

For a fixed $p\in X$, choose a Weyl chamber $W\subset S_pX$. Then for each $v$ in the interior of $W$ the orbit $Gv$ constitutes a closed set invariant under the geodesic flow. The geodesic flow restricted to the submanifold $SM_v:= \Gamma\setminus Gv \subset SM$ is mixing with respect to the probability measure $m^v$ induced by the Liouville measure $m$ on $SM$. Moreover, there exists a unique $v_{\text{max}}$ such that $h_\top(g^t)=h_\top(g^t, SM_{v_{\text{max}}})$ and $g$ restricted to $SM_{v_{\text{max}}}$ is a $C^\infty$ PHD. See \cite{Kni} for more details. At any point in $SM_{v_{\text{max}}}$, the unstable manifolds of $g|SM_{v_{\text{max}}}$ coincides with the unstable manifolds of $g$. However, $g|SM_{v_{\text{max}}}$ is non-expanding (and hence of bounded expansion) in the center direction now, i.e.,
$$d^{cs}(g^n(x), g^n(y))\le d^{cs}(x,y)$$
for any $x,y$ on the same local center-unstable leaf (see Page 181 in \cite{Kni}). On the other hand, we know that $g^t$ is topologically transitive on $SM_{v_{\text{max}}}$ as it is ergodic with respect to $m^{v_{\text{max}}}$. One can check that Lemma \ref{uniformtran} holds for a real number $0\leq k\leq n$ (see the proof of Proposition \ref{ueg0}). Furthermore, the proof of Proposition \ref{ueg0} goes through with real $k$. Thus the conclusion of Theorem B applies to $g$.

\section{PHDs with subexponetial growth in the center direction}
In this section, we consider various PHDs with subexponential growth in the center direction. Recall that a $C^1$ PHD $f: M \to M$ has subexponential growth in the center direction if for any $\e>0$, there exists a constant $C_\e>0$ such that $\|D_xf^n|_{E_f^c(x)}\|\leq C_\e e^{n\e}$ for any $n\in \mathbb{N}$ and any $x\in M$. The following lemma is immediate.
\begin{lemma}\label{growth}
If $f: M \to M$ is a $C^1$ PHD with subexponential growth in the center direction, then for any $x\in M$, $\d,\e>0$ and $n\in \NN$, $$f^n(W_f^{cs}(x,C_\e^{-1} e^{-n\e}\d/2))\subset W^{cs}_f(f^nx,\d).$$
\end{lemma}

\subsection{PHDs with property (EM)}
In the first subsection, we prove Theorem C. Let $f: M \to M$ be a $C^1$ dynamically coherent partially hyperbolic diffeomorphism which has subexponential growth in center direction and satisfies property (EM).

Given any $x_0\in M, 0<\d<\d_0$ and $\e>0$, $R(x_0,n,\e,\d)$ is called an \emph{$(n,\e,\d)$-rectangle about $x_0$} if it is in the form
\begin{equation*}
R(x_0,n,\e,\d):=\bigcup_{y\in W_f^{cs}(x_0,C_\e^{-1} e^{-n\e}\d/2)}W^u_f(y,2\d).
\end{equation*}

As $W^u$ and $W^{cs}$ are uniformly transverse continuous foliations, there exists $c_0=c_0(\d_0)>0$ such that for any $0<\d<\d_0$,
\begin{equation}\label{e:contains}
B(x_0,c_0^{-1}C_\e^{-1} e^{-n\e}\d/2)\subset R(x_0,n,\e,\d)\subset \bigcup_{y\in W^u_f(x_0,3\d)}W_f^{cs}(y,c_0C_\e^{-1} e^{-n\e}\d/2).
\end{equation}

Recall that for a compact Riemannian manifold $M$, $\phi\in C^{\infty}(M,\RR)$ and $l\in \NN$, the Sobolev norm is defined as
\begin{equation}\label{e:sobolev}
\|\phi\|_{l,2}=\left(\sum_{0\le j \le l} \int_M |\nabla^j \phi|^2d\mu\right)^{1/2},
\end{equation}
where $\nabla$ denotes the covariant derivative and $\mu$ denotes the volume measure associated to $M$ (cf. \cite[Definition 2.1]{Eh}).

The following is lemma is useful. See \cite[Lemma 3.3]{Kad} for proof.
\begin{lemma}\label{appr}
Let $M$ be a $d$-dimensional compact smooth Riemannian manifold and $l\in \NN$. Then there exist $r_0>0$ and $M_{l,d}>0$ such that the following holds: for any $x\in M$ and $0<\epsilon, r<r_0$, there exists $\phi_\e\in C^\infty(M,[0,1])$ such that
\begin{enumerate}
  \item $\phi_\e \equiv 1$ on $B(x,r)$;
  \item $\phi_\e \equiv 0$ on $B(x,r+\e)^c$;
  \item $\|\phi_\e\|_l\le M_{l,d}\e^{-d-l}$;
  \item $\|\phi_\e\|_{l,2}\le M_{l,d}\e^{-d-l}$.
\end{enumerate}

\end{lemma}

The following is a crucial lemma on effective density of the unstable foliation.
\begin{lemma}\label{l-main}(Effective density)
For any $0<\d<\d_0$ and $\e>0$, there exist $N_1\in \NN$ and $D>0$ such that for any $n\ge N_1$,  $k\ge Dn\e$ and any $x_0,x \in M$, $f^kW^u_f(x,\d)\cap R(x_0,n,\e,\d)$ has a connected component of the form $W^u_f(y,2\d)$ for some $y\in W_f^{cs}(x_0,C_\e^{-1} e^{-n\e}\d/2)$.
\end{lemma}

\begin{proof}
Fix arbitrary $0<\d<\d_0$ and $\e>0$. At first, we show that there exists $N_1'\in \NN$ and $D>0$ such that for any $n\ge N_1'$,  $k\ge Dn\e$ and any $x,x_0 \in M$,
\begin{equation}\label{intsec-ball}
f^kW^u_f(x,\d/2) \cap B(x_0,c_0^{-1}C_\e^{-1} e^{-n\e}\d/2) \ne \emptyset
\end{equation}
where $c_0$ satisfies \eqref{e:contains}.

Let $A=W^u_f(x,\d/2), B=B(x_0,c_0^{-1}C_\e^{-1} e^{-n\e}\d/2)$, $\chi_A$ and $\chi_B$ their characteristic functions.
Applying Lemma \ref{appr} for $B(x,r)=A$ (resp. $B(x,r)=B$), $\e=e^{-\lambda' k}$ for some sufficiently small $\lambda'$ specified later,
we obtain $C^\infty$ functions $\psi_k$ (resp. $\phi_k$).

Applying property (EM) \eqref{e:mixing0} to $C^\infty$ functions $\phi_k$ and $\psi_k$, we have
\begin{equation}\label{e:eff1}
\begin{aligned}
&\int_{W_f^u(x,\d)}\phi_k(f^k(p))\psi_k(p)dm_x^u(p)\ge \\
&\int\phi_k dm\int_{W_f^u(x,\d)}\psi_k dm_x^u-Ce^{-k\a}M_{l,d}M_{l,u}e^{\lambda'k(2l+d+u)}
\end{aligned}
\end{equation}
for any $k\ge 0$ where $u=\dim W_f^u(x,\d)$. By choosing $\lambda'$ small enough such that $\a-(2l+d+u)\lambda'>\lambda'$, we see from \eqref{e:eff1} that
\begin{equation*}
\begin{aligned}
\int_{W_f^u(x,\d)}\phi_k(f^k(p))\psi_k(p)dm_x^u(p)\ge \int\phi_k dm\int_{W_f^u(x,\d)}\psi_k dm_x^u-C_1e^{-\lambda'k}.
\end{aligned}
\end{equation*}
Then we have
\begin{equation*}
\begin{aligned}
&\int_{W_f^u(x,\d)}\chi_B(f^k(p))\chi_A(p)dm_x^u(p)\\
\ge &\int_{W_f^u(x,\d)}\phi_k(f^k(p))\psi_k(p)dm_x^u(p)-\int |\psi_k-\chi_A|dm_x^u-\int |\phi_k-\chi_B|dm\\
\ge &\int\phi_k dm\int_{W_f^u(x,\d)}\psi_k dm_x^u-C_1e^{-\lambda'k}-C_2e^{-\lambda'k}\\
\ge &C_3(\d/2)^u\cdot (c_0^{-1}C_\e^{-1} e^{-n\e}\d/2)^d-C_4e^{-\lambda'k}\\
=&C_5e^{-nd\e}-C_4e^{-\lambda'k}>0
\end{aligned}
\end{equation*}
if $n\ge N_1'$,  $k\ge Dn\e$ for some sufficiently large $N_1'\in \NN$ and some $D>\frac{d}{\lambda'}$.
Then \eqref{intsec-ball} follows immediately.

Now set $N_1=\max\{N_1',N_1''\}$, where $N_1''$ satisfies that, for any $k\ge DN_1''\e$,
\begin{equation*}
f^{-k}W_f^u(w,4\d)\subset W^u_f(f^{-k}w,\d/2)
\end{equation*}
for any $w\in M$.
Take $n\ge N_1$ and $k\ge Dn\e$.
In view of \eqref{e:contains} and \eqref{intsec-ball}, there exists
\begin{equation}\label{e:pointz}
\begin{aligned}
z\in f^kW^u_f(x,\d/2) \cap B(x_0,c_0^{-1}C_\e^{-1} e^{-n\e}\d/2)\subset f^kW^u_f(x,\d/2) \cap R(x_0,n,\e,\d).
\end{aligned}
\end{equation}
Therefore, there exists $y\in W_f^{cs}(x_0,C_\e^{-1} e^{-n\e}\d/2)$ such that $z\in W^u_f(y,2\d)$.
Take any $w\in W^u_f(y,2\d)$, and it is clear that
$d^u(z,w)\le 4\d.$
Then by the choice of $N_2''$, $d^u(f^{-k}w,f^{-k}z)\le \d/2$. In view of \eqref{e:pointz}, we have
$$f^{-k}w\in W^u_f(x,\d).$$
This implies $W^u_f(y,2\d)\subset f^k W^u_f(x,\d) \cap R(x_0,n, \e,\d)$ and completes the proof of the lemma.
\end{proof}

\begin{proposition}\label{uegnegeral}
Let $f: M \to M$ be a $C^1$ dynamically coherent PHD which has subexponential growth in the center direction and satisfies property (EM). Then $f$ has uniform exponential growth along unstable foliation.
\end{proposition}

\begin{proof}
Denote by $h$ the unstable topological entropy of $f$, i.e., $h=h^u_{\top}(f)$. Since $f$ has subexponential growth in the center direction,
Proposition \ref{unstable} tells us that $h=h_{\top}(f)$, the topological entropy of $f$.
Then by the definition of unstable topological entropy and Remark \ref{unstableentropy}, for any $\rho>0$, we can find $x_0\in M$ and $0<\d<\d_0$ small enough such that
$$h_{\top}(f, \overline{W^u_f(x_0,\d)})>h-\rho.$$
Pick $0<\e\ll \d$ and $N_2\in \NN$ large enough such that if $n\ge N_2$, then
$$N^u(f,n,\e, x_0,\d)>e^{n(h-2\rho)},$$
where $N^u(f,n,\e, x_0,\d)$ denotes the cardinality of a maximal $(n,\e)$ u-separated subset $S$ of $\overline{W^u_f(x_0,\d)}$, i.e.,
$\#S=N^u(f,n,\e, x_0,\d)>e^{n(h-2\rho)}$ where $\#S$ denotes the cardinality of $S$.

We note that by dynamical coherence of $f$, the center-stable foliation $W_f^{cs}$ exists and it is continuous. Shrinking $\d_0$ if necessary, for any $y\in W^{cs}_f(x,c_0\d_0)$, the $cs$-holonomy map $H_{x,y}^{cs}$ is well defined such that
$$H^{cs}_{x,y}(\overline{W^u_f(x,\d)})\subset W^u_f(y,c_2(\d))$$
and
\begin{equation*}
H^{cs}_{x,y}(W^u_f(x,\e/2))\supset W^u_f(y,c_3(\e,\d))
\end{equation*}
for some constants $c_2=c_2(\d)>2\d$ and $c_3=c_3(\e,\d)>0$.
Then take $N_3\in \NN$ large enough such that for any $n\ge N_3$ and for any $x\in M$:
\begin{enumerate}
  \item $f^n(W^u_f(x,c_3))\supset W^u_f(f^n(x),2\d)$;
  \item $f^n(W^u_f(x,\d))\supset W^u_f(f^n(x),2c_2)$.
\end{enumerate}


Now take any $x\in M$ and $\bar{\e}$ small enough to be specified later. Applying Lemma \ref{l-main} with $x=f^{N_3}x$, $x_0=x_0$, $\d=\d$ and $\e=\bar{\e}$, there exists $N_1\in \NN$ depending only on $\d$ and $\bar{\e}$ such that for any $n\ge N_1$
\begin{equation}\label{e:intersection}
f^k(W_f^u(f^{N_3}x,\d))\cap W_f^{cs}(x_0,C_{\bar\e}^{-1} e^{-n\bar\e}\d/2)\neq \emptyset
\end{equation}
where $k=Dn\bar\e$.

Finally, we take $\bar\e$ small enough and $n_\rho\in \NN$ large enough such that $n_\rho>\max\{N_1, N_2, N_3\}$ and
$$\frac{n_\rho}{2N_3+Dn_\rho\bar\e+n_\rho}>\frac{h-3\rho}{h-2\rho}.$$

The remaining part of the proof of the proposition is divided into three steps:

\textbf{Step $1$.} Denote $k=Dn_\rho\bar\e$. It follows from \eqref{e:intersection} that there exists some point
$$y\in f^{k+N_3}(W_f^u(x,\d))\cap W_f^{cs}(x_0,C_{\bar\e}^{-1} e^{-n\bar\e}\d/2).$$
By the choice of $N_3$ and the fact $c_2>2\d$, we have
$$f^{k+N_3}(W_f^u(x,\d))\supset W^u_f(y,c_2).$$
Hence the cs-holonomy map
$$H^{cs}_{x_0,y}: \overline{W^u_f(x_0,\d)}\to W^u_f(y,c_2)\subset f^{k+N_3}(W_f^u(x,\d))$$
is well defined by the choice of $c_2$.

\textbf{Step $2$.} Let $S$ be an $(n_\rho,\e)$ u-separated subset of $\overline{W^u_f(x_0,\d)}$ of maximal cardinality. Denote $S'=H^{cs}_{x_0,y}(S)$.
By \eqref{e:contains}, we have
$$d_f^{cs}(x_i,y_i)\le c_0C_{\bar\e}^{-1} e^{-n\bar\e}\d/2$$
for any $x_i\in S$ and $y_i=H^{cs}_{x_0,y}(x_i)\in S'$.

Since $S$ is $(n_\rho,\e)$ u-separated, we know that
$$\{W_f^u(f^{n_\rho}x_i, \e/2): x_i\in S\}$$
are disjoint balls contained in $f^{n_\rho}\overline{W^u_f(x_0,\d)}$.
By Lemma \ref{growth},
$$d^{cs}(f^{n_\rho}x_i,f^{n_\rho}y_i)< c_0\d.$$
Then we have
$$H^{cs}_{f^{n_\rho}x_i,f^{n_\rho}y_i}(W_f^u(f^{n_\rho}x_i, \e/2))\supset W^u(f^{n_\rho}y_i, c_3)$$
by the definition of $c_3$. Moreover,
$$\{W_f^u(f^{n_\rho}y_i, c_3): y_i\in S'\}$$
are disjoint sets contained in $f^{n_\rho}W^u_f(y,c_2)$.

\textbf{Step $3$.} By the choice of $N_3$,
$$f^{N_3}(W^u_f(f^{n_\rho}y_i,c_3)\supset W^u_f(f^{N_3+n_\rho}y_i,2\d).$$
It follows that
$$\{W^u_f(f^{N_3+n_\rho}y_i,2\d): y_i\in S'\}$$
are disjoint sets contained in $f^{N_3+n_\rho}W^u_f(y,c_2)\subset f^{n_\rho+Dn_\rho\bar\e+2N_3}W^u_f(x,\d)$.
Recall that $\bar\e$ is depending on $\rho$. Set
$$N(\rho, \d)=n_\rho+Dn_\rho\bar\e+2N_3.$$
Since $\#S'=\#S>e^{n_\rho(h-2\rho)}>e^{N(\rho, \d)(h-3\rho)}$, we know
$f^{N(\rho, \d)}W^u_f(x,\d)$ contains at least $e^{N(\rho, \d)(h-3\rho)}$ disjoint sets
$$\{W^u_f(f^{N_3+n_\rho}y_i,2\d): y_i\in S'\}$$
This proves the proposition by setting $z_i=f^{N_3+n_\rho}y_i$.
\end{proof}

\begin{proof}[Proof of Theorem C]
Theorem C is an immediate consequence of Proposition \ref{uegnegeral} and Theorem A.
\end{proof}
\subsection{Homogeneous dynamics}
\subsubsection{Definitions and basic properties}
Let $G$ be a connected semisimple Lie group, and $g\in G$. For $\lambda \in \mathbb{C}$, set
$$E_\lambda:=\{X\in \fg_\CC: (\Ad(g) -\lambda I)^jX=0 \text{\ for some\ } j\in \mathbb{N}-\{0\}\},$$
where $\fg_\CC$ denotes the complexification of $\fg$.
Let $\fh^+, \fh^-$ be the subalgebras of $\fg$ with complexifications
\begin{equation}\label{e:spaces}
\fh^+_\CC:=\bigoplus_{|\lambda|>1} E_\lambda, \quad \fh^-_\CC:=\bigoplus_{|\lambda|\le 1} E_\lambda,
\end{equation}
and let $H^+,  H^-$ be the corresponding subgroups of $G$. Then $H^+$ is the unstable horospherical subgroup for $g$, that is,
 $$H^+=\{h\in G: \lim_{k\to +\infty}g^{-k}hg^k=e\}$$
where $e\in G$ is the identity of $G$. $H^+$ is a closed, connected subgroup of $G$.
Then $\Ad g$ (or just $g$ for short) is called \emph{non-quasiunipotent} if $H^+$ is nontrivial.
Let $d^+$(resp. $d^-$) be the right invariant metric on $H^+$(resp. $H^-$) induced from the norm on $\fh^+$(resp. $\fh^-$) inherited from that of $\fg$. Let $B_r^+$(resp. $B_r^-$) denote the corresponding metric ball in $(H^+,d^+)$ (resp. $(H^-,d^-)$) centered at the identity $e\in G$ with radius $r$.

Let $\Gamma$ be a cocompact lattice in $G$, then we have a compact homogenous space $\ggm$. By compactness, there exist $\d_0>0$ and $c_0>1$ such that for any $0<\d<\d_0$ and any $x\in \ggm$, the map $$\psi_x:B_\d^+\times B_\d^-\rightarrow \ggm, \quad \psi_x(h^+,h^-)=h^+h^-x,$$ is a diffeomorphism onto its image, and moreover
\begin{equation*}\label{e:contain}
B(x,c_0^{-1}\d)\subset B_\d^+ B_\d^-x\subset B(x,c_0\d).
\end{equation*}

Denote the conjugate action of $g$ on $G$ as $C(g)$. Then $C(g)$ defines an expanding automorphism on $H^+$, and an at most polynomially expanding automorphism of $H^-$. Below is an explicit description of the action of $C(g)$ on $H^+$ and $H^-$.
\begin{lemma}\label{l-action-g}
\
\begin{enumerate}
\item There exist $\lambda\in (0,1)$ and $N_1\in \NN$ such that
$C(g^{-n})h\in B_{\lambda^n\d}^+$ whenever $n\ge N_1, 0<\d\le \d_0$ and $h\in B_\d^+$.
\item There exist $c_1,\kappa>0$ such that $C(g^{i})(B_{c_1n^{-\kappa}\d}^-)\subset B_\d^-$ when $0<\d\le \d_0$ and $1\le i\le n$.
\end{enumerate}
\end{lemma}

\begin{proof}
All these properties are easily checked by passing to the Lie algebras $\fh^+, \fh^-$ using the exponential maps.
\end{proof}

Let $T_g: \ggm\to \ggm$ be the dynamical system induced by the left translation of $T_g(x)=gx$, for any $x\in \ggm$. Assume that $\Ad g$ is nonquasiunipotent. Then $T_g$ is a $C^\infty$ partially hyperbolic diffeomorphism. The local unstable and center-stable manifolds of $T_g$ at $x\in \ggm$ of size $\d>0$ are given by
$$W^u_g(x,\d)=B_\d^+x, \quad W^{cs}_g(x,\d)=B_\d^-x.$$
In the following, we will use both notations $W^u_g(x,\d)$ and $B_\d^+x$ (resp. $W^{cs}_g(x,\d)$ and $B_\d^-x$) interchangeably for the local unstable (resp. local center-stable) manifold of $T_g$ at $x\in \ggm$. $T_g$ has polynomial (hence subexponential) growth in the center direction, by Lemma \ref{l-action-g} and the discussion before it.

The topological entropy of $T_g$ can be calculated by the formula:
$$h_\top(g)=h_{\mu_G}(g)=\log |\det (\Ad g|_{\fh^+})|=\sum_{|\lambda|>1}\log |\lambda|$$
where the sum is taken over all the eigenvalues of $\Ad(g)|_{\fh^+}$ counted with multiplicity, and $\mu_G$ is the normalized Haar measure on $\ggm$. In fact, the above formula holds in a more general setting when $\Gamma$ is any lattice of $G$
(cf. \cite{Bow} or Corollary 2.5.7 in \cite{Mo}).

\subsubsection{Exponential mixing}
Now we recall the exponential mixing property of $T_g$. To begin, let us recall that for a Riemannian manifold $M$, $f\in C_c^{\infty}(M,\RR)$ and $l\in \NN$, the Sobolev norm is defined in \eqref{e:sobolev}.

\begin{proposition}\label{prop-km}(Exponential mixing)
Let $G$ be a connected semisimple Lie group without compact factors, $\Gamma$ be an irreducible cocompact lattice of $G$ and $g\in G$ be  non-quasiunipotent. Then there exist $\mu, E, l, \theta > 0$ such that for any $x \in \ggm$, $\psi \in C^\infty(\ggm,\RR)$ and $f \in C^\infty_c (H^+,\RR)$ such that the map $\pi_x:g\rightarrow gx$ is injective on $\text{supp}(f)$, and for any $n\in \NN$ one has
\begin{equation*}
\left|\int_{H^+} f(h)\psi(g^nhx)d\nu_{H^+}(h)-\int_{H^+} f d\nu_{H^+}\int_{\ggm} \psi d\mu_G\right| \leq C(f,\psi)e^{-\mu n},
\end{equation*}
where $\nu_{H^+}$ is a bi-invariant Haar measure on $H^+$, $\mu_G$ is the normalized Haar measure on $\ggm$, and
\begin{equation*}
C(f,\psi)=E\|f\|_{l,2}\|\psi\|_{l,2}\left(\max_{x\in \ggm}\|\nabla \psi(x)\|\int_{H^+} |f|d\nu_{H^+}\right)^{\theta}.
\end{equation*}
\end{proposition}

In \cite[Proposition 2.4.8]{KM}, the authors obtained the same result for $(\ggm, g^t)$ where $g^t$ is a one-parameter subgroup of $G$, which is derived from the Condition (EM) \cite[Section 2.4.1]{KM}. In our discrete setting, the authors in \cite[Propositions 3.1]{GSW} provide a proof for a spectral gap result for the regular representation of $G$ on the space $L_0^2(\ggm)$, from which the Condition (EM) follows.

We point out that the coefficient $C(f,\psi)$ is slightly different from that in \eqref{e:mixing0}. Nevertheless, Lemma \ref{l-main} can be proved analogously. See also the proof of \cite[Proposition 3.2]{Kad}, \cite[Proposition 3.3]{BK} and \cite[Lemma 3.5]{GSW} where indeed $C(f,\psi)$ is used.

Thus Propositions \ref{uegnegeral} and \ref{prop-km} also hold. As a corollary, we have:
\begin{proposition}\label{special}
Let $G$ be a connected semisimple Lie group without compact factors, $\Gamma$ be an irreducible cocompact lattice of $G$ and $g\in G$ be non-quasiunipotent. Then $(\ggm,g)$ has uniform exponential growth along unstable foliation.
\end{proposition}

Proposition \ref{special} can be extended to a more general setting.
\begin{proposition}\label{general}
Let $G$ be a connected semisimple Lie group, $\Gamma$ be a cocompact lattice of $G$, and $g\in G$ be non-quasiunipotent. Then $(\ggm,g)$ has uniform exponential growth along unstable foliation.
\end{proposition}

\begin{proof}
Let $\widetilde{G}$ be the simply connected cover of $G$, $\widetilde{\Gamma}$ be the inverse image of $\Gamma$ in $\widetilde{G}$ and $\widetilde{g}$ be any inverse of $g$ in $\widetilde{G}$. Then it is easily seen that there is an isomorphism of systems
$$(\widetilde{G}/\widetilde{\Gamma},\widetilde{g})\cong (\ggm,g).$$
Hence we may replace $G$ by its simply connected cover $\widetilde{G}$ and $\Gamma$ by its inverse image in  $\widetilde{G}$.
Write $$G=G^{nc}\times K, \text{ with $K$ compact and $G^{nc}$ has no compact factors. }$$
Then $q(\Gamma)$ is a lattice in $G^{nc}$, where $q$ denotes the projection from $G$ to $G^{nc}$. According to \cite[Theorem 5.22]{Rag}, $G^{nc}$ can be decomposed as
$$G^{nc}=G_1\times\cdots\times G_m$$
such that $\Gamma\cap G_i$ are irreducible lattices in $G_i (1\le i\le m)$.
Write $g=(g_1,\ldots,g_m,k)$ with $g_i\in G_i$ and $k\in K$. We may assume that $g_i$ is non-quasiunipotent for $1\le i \le j$ and $g_i$ is quasiunipotent for $j<i \le m$. Then set $$G'=G_1\times\cdots\times G_j, \quad G''=G_{j+1}\cdots\times G_m\times K.$$
By assumption $G'$ is nontrivial and we write $g=(g',g'')$ with $g'\in G'$ non-quasiunipotent and $g''\in G''$ quasiunipotent.
Denote the projection of $G$ to $G'$ as $p$. Using \cite[Theorem 5.22]{Rag} again, we obtain that the group $p(\Gamma)$ is commensurable with $\prod_{1\le i\le j}(\Gamma\cap G_i)$, hence is a lattice in $G'$. Thus we know that $\ggm$ is a fiber bundle over $G'/p(\Gamma)$ with fiber $G''/(\Gamma\cap G'')$.

Firstly, we study the system $(G'/p(\Gamma),g')$. By definition, there is a finite index subgroup $\Gamma'\subset p(\Gamma)$ such that $\Gamma'=\prod_{1\le i\le j}\Gamma_i$ with $\Gamma_i$ irreducible lattices in $G_i$. By Proposition \ref{special}, each $(G_i/\Gamma_i,g_i)$ has uniform exponential growth along unstable foliation. Then $(G'/\Gamma',g',d')$ is a product of systems $(G_i/\Gamma_i,g_i,d_i)$ where it is more convenient to work with the metrics
$$d'((x_i),(y_i))=\max_{1\le i\le j} d_i(x_i,y_i).$$
Note that $h_{\top}(g')=\sum_{1\le i\le j}h_{\top}(g_i)$. It is not hard to verify that $G'/\Gamma'$ has uniform exponential growth along unstable foliation by choosing
$$N(\rho,\d)=\max_{1\le i \le j} N_i(\rho,\d).$$
Now taking into account that there exists a locally isometric finite covering map $\varphi: G'/\Gamma'\to G'/p(\Gamma)$, we see that  $G'/p(\Gamma)$ has uniform exponential growth along unstable foliation.

Now consider the fiber system $(\ggm, g)$ over $(G'/p(\Gamma),g')$ with fiber $G''/(\Gamma\cap G'')$. Note that $h_{\top}(g)=h_{\top}(g')$ and the Lyapunov exponents of $g$ along the direction of fibers are zero. Notice that the unstable horospherical subgroup $H^+$ of $g$ is a subgroup of $G'$. If $x=h\Gamma=h'h''\Gamma\in \ggm$ with $h'\in G'$ and $h''\in G''$, then for $B^+_\d\subset H^+$,
$$p(B_\d^+h\Gamma)=B_\d^+p(h\Gamma)=B_\d^+h'p(\Gamma).$$
So $p$ is in fact a local isometry between $W^u_g(x)\to W^u_{g'}(p(x))$. As $p\circ T_g=T_{g'}\circ p$, it is clear that $(\ggm, g)$ has uniform exponential growth along unstable foliation if $(G'/p(\Gamma),g')$ has.
\end{proof}

\begin{proof}[Proof of Theorem D]
If $g$ is quasiunipotent, the topological entropy of $T_g$ is $0$, then there is nothing to prove.

Assume that $g$ is non-quasiunipotent. Then by Proposition \ref{general}, $(\ggm,g)$ has uniform exponential growth along unstable foliation. Moreover, it has subexponential growth in the center direction. Now Theorem D is an immediate corollary of Theorem A(2).
\end{proof}

\subsubsection{An improvement of Theorem D}
\begin{TheoremD'}
Let $G$ be a connected semisimple Lie group, $\Gamma$ be a cocompact lattice of $G$, and $g\in G$. Then the topological entropy function $h_{\top}: C^1(\ggm)\to \RR^+\cup\{0\}$ is continuous at $T_g$.
\end{TheoremD'}

\begin{proof}
Let $\e>0$ and $f$ be $C^1$-close enough to $T_g$ which will be specified later. By Ruelle inequality, for any $f$-ergodic measure $\nu$ and any $N\in \NN$,
\begin{equation}\label{e:rue}
\begin{aligned}
h_\nu(f)&\le \sum_{\lambda_j(\nu)>0}\lambda_j(\nu)m_j(\nu)\\
&\le\dim E_f^c \cdot \frac{1}{N}\max_{x\in \ggm} \log\|D_xf^N|_{E_f^c(x)}\|+\max_{x\in \ggm} \log\det (D_xf|_{E_f^u(x)}).
\end{aligned}
\end{equation}
Note that $\|D_xT_g^N|_{E_g^c(x)}\|$ is constant in $x\in \ggm$. Choose $N$ large enough such that $\frac{1}{N}\log \|D_xT_g^N|_{E_g^c(x)}\|< \frac{\e}{3\dim E_g^c}$ for any $x\in \ggm$. Now choose $f$ $C^1$-close enough to $T_g$ such that for any $x\in \ggm$,
\begin{equation}\label{e:close1}
\Big| \log \|D_xf^N|_{E_f^c(x)}\|- \log\|D_xT_g^N|_{E_g^c(x)}\|\Big|< \frac{\e}{3\dim E_f^c}
\end{equation}
and
\begin{equation}\label{e:close2}
\Big| \log \det(D_xf|_{E_f^u(x)})-\log \det(D_xT_g|_{E_g^u(x)})\Big|< \frac{\e}{3}.
\end{equation}
Note that
\begin{equation}\label{e:close3}
h_\top(T_g)=\log\det (\Ad g|_{\fh^+})=\log\det(D_xT_g|_{E_g^u(x)})
\end{equation}
for any $x\in \ggm$. Combining \eqref{e:rue}, \eqref{e:close1}, \eqref{e:close2} and \eqref{e:close3},
we have $h_\nu(f)< h_\top(T_g)+\e$. Then $h_\top(f)=\sup_{\nu\in \M_f^e(\ggm)}h_\nu(f)\le h_\top(T_g)+\e$. This proves the upper semicontinuity of $h_{\top}: C^1(\ggm)\to \RR^+\cup\{0\}$ at $T_g$.

The lower semicontinuity of $h_{\top}: C^1(\ggm)\to \RR^+\cup\{0\}$ at $T_g$ is obtained in Theorem D. We give a simple proof when the system after perturbation is $C^r, r>1$. For any $\e>0$, choose a $C^r,r>1$ diffeomorphism $f:\ggm\to \ggm$ which is $C^1$-close enough to $T_g$ such that for any $x\in \ggm$ \eqref{e:close2} is satisfied. As $f$ is $C^r,r>1$, let $\mu$ be a Gibbs u-state for $f$, i.e., an $f$-invariant measure whose conditional measures along unstable leaves are absolutely continuous with respect to the Lebesgue measure on the unstable leaves. The existence of $\mu$ is guaranteed by \cite{PS0}. Then we have
\begin{equation}\label{e:entropyesti}
\begin{aligned}
&h_\top(f)\ge h_\top^u(f)\geq h^u_\mu(f)
\\=&\int_{\ggm} \log \det(D_xf|_{E_f^u(x)})d\mu(x)\ge \log \det(DT_g|_{E_g^u})-\e/3\\
=&h_\top(T_g)-\e/3.
\end{aligned}
\end{equation}
This proves the lower semicontinuity of $h_{\top}: C^r(\ggm)\to \RR^+\cup\{0\}$ at $T_g$, $r>1$.
\end{proof}
\begin{remark}
\eqref{e:close3} is crucial for the above argument to give upper semicontinuity. In general, PHDs considered in Theorem C do not necessarily satisfy \eqref{e:close3}. As far as we know, a Gibbs u-state satisfying the entropy estimate in \eqref{e:entropyesti} exists only for $C^r,r>1$ PHDs.
\end{remark}
\subsubsection{Further questions}
Though the most work in this paper is done on the unstable foliation of $T_g$, we remark that the center distribution of $T_g$ corresponding to the subalgebra $\fh^0:=\bigoplus_{|\lambda|= 1} E_\lambda\cap \fg$, is $C^\infty$. Then the center foliation is $C^\infty$ and hence plaque expansive (cf. \cite{HPS}). Thus there exists a neighborhood $\mathcal U$ of $T_g$ in the $C^1$ topology, such that for each $f\in \mathcal U$,
\begin{enumerate}
  \item there exists a homeomorphism $h_f: \ggm \to \ggm$, called a leaf conjugacy, satisfying $h_f(W^c_g(x))=W^c_f(h_f(x))$ and $h_f\circ T_g(W^c_g(x))=f\circ h_f(W^c_g(x))$. Moreover, the leaf conjugacy $h_f$ is close to the identity map in the $C^0$ topology.
  \item $f$ is dynamically coherent, i.e., there exist $f$-invariant foliations $W^{cs}_f$ and $W^{cu}_f$ tangent to $E^c_f\oplus E^s_f$ and $E^c_f\oplus E^u_f$ respectively; by intersecting their leaves we then obtain an $f$-invariant foliation $W^c_f$.
\end{enumerate}
A further question arises about the continuity property of the topological entropy at each $f\in \mathcal U$. Clearly, the properties of $f$ listed above will be fundamental. However, other properties of $f$ such as property (EM) should be developed, for which we have no explicit ideas now.

On the other hand, one can study a smooth time change of a partially hyperbolic homogeneous flow.
Let $(\ggm,g_t)$ be a partially hyperbolic homogeneous flow, i.e., $g_t$ is a one-parameter subgroup of $G$ for which each element is non-quasiunipotent.
Let $\tau: \ggm\to  \RR$ be a smooth positive function. The time-changed flow is defined as
$g^\tau_t(x):=g_{\a(x,t)}(x)$,
where $\a(x,t)$ is determined by
$$ \int_0^{\a(x,t)}\tau(g_s(x))ds=t.$$
Then the time-one map $g^\tau_1$ is also a $C^\infty$ PHD whose center, center-stable and center-unstable foliations all coincide with those of $T_g$. However, we do not know if $g^\tau_1$ has property (EM).

\subsection{Toral automorphisms}
Suppose that a toral automorphism $T_A:\mathbb{T}^d\to \mathbb{T}^d$ is induced by a linear map $A:\mathbb{R}^d\to \mathbb{R}^d$. For $\lambda \in \mathbb{C}$, set
$$E_\lambda:=\{v\in \CC^d: (A-\lambda I)^jv=0 \text{\ for some\ } j\in \mathbb{N}-\{0\}\}.$$
The so-called stable, center and unstable subspaces $E^u,E^c,E^s$ are the intersections with $\RR^d$ of
\[E^u_\CC:=\bigoplus_{|\lambda|>1} E_\lambda, \quad E^c_\CC:=\bigoplus_{|\lambda|= 1} E_\lambda, \quad E^s_\CC:=\bigoplus_{|\lambda|<1} E_\lambda\]
respectively. Then $\mathbb{R}^d=E^u\oplus E^c\oplus E^s$. Denote $E^{sc}=E^s\oplus E^c$. We can use the supremum norm $\|\cdot\|_\infty$ (denoted by $\|\cdot\|$ for short) on $\mathbb{R}^d$ defined by $\|v\|:=\max\{\|v^u\|, \|v^c\|, \|v^s\|\}$ where $v=v^u+v^c+v^s, v^*\in E^*, *\in \{s,c,u\}$. Denote $B^*(\e):=\{v\in E^*: \|v\|<\e\}, *\in \{s,c,u\}$.

Assume that $T_A:\mathbb{T}^d\to \mathbb{T}^d$ is ergodic. It is enough to consider the nontrivial case $E^u\neq \emptyset$. The following lemma is well known.
\begin{lemma}\label{dense}(Effective density, see Lemma 1 in \cite{Mar})
Let $\tau>1$. Then for $n\in \NN$ large enough, $B^u(\tau^n)$ is $(1/n^d)$-dense in $\mathbb{T}^d$.
\end{lemma}
It then follows form Lemma \ref{dense} that there exists $N_1\in \NN$ such that for any $n\ge N_1$, we have an integer function $M=M_\d(n)$ such that for any $x,x_0\in \mathbb{T}^d$,
$$T_A^M(x+B^u(\d/2))\cap \Big(x_0+B(\frac{\d}{2n^{d-2}})\Big) \neq \emptyset$$
and $\lim_{n\to \infty}\frac{M_\d(n)}{n}=0$. Then $T_A^M (x+B^u(\d))$ contains a local unstable leaf $W^u_{T_A}(y,2\d)$ for some $y\in W_{T_A}^{cs}(x_0,\frac{\d}{2n^{d-2}})$. We note that
 \begin{itemize}
   \item  $d(T_A^nx,T_A^ny)\le \d$ if $y\in W_{T_A}^{cs}(x,\frac{\d}{n^{d-2}})$,
   \item  $W^{cs}$ is a $C^\infty$ foliation,
   \item  $\lim_{n\to \infty}\frac{M_\d(n)}{n}=0$.
 \end{itemize}
The proof of Proposition \ref{uegnegeral} still works here with minor modifications and thus Theorem E also follows from Theorem A.

We remark that Theorem E can be improved as follows.
\begin{TheoremE'}\label{constant}
Suppose that $T_A:\mathbb{T}^d\to \mathbb{T}^d$ is ergodic with $E^u\neq \emptyset$. Then $h_\top^u: C^1(\mathbb{T}^d)\to \RR^+\cup\{0\}$ is locally constant at $T_A$, and $h_\top: C^1(\mathbb{T}^d)\to \RR^+\cup\{0\}$ is continuous at $T_A$.
\end{TheoremE'}

\begin{proof}
The first part of Theorem E' can be derived from the proof of Theorem 1.1 in \cite{HSX}. The authors show that the unstable foliation of $T_A$ stably carries a unique non-trivial homology and hence obtain the above result, taking into account the fact $\chi_u(T_A)=h_\top^u(T_A)$ by \cite{HHW}. The second part of Theorem E' can be proved similarly as in the proof of Theorem D' in Section 5.2.3.
\end{proof}

\begin{remark}
The technique in \cite{HSX} is inapplicable to homogeneous translations. Indeed, let $g^t, t\in \RR$ be a one-parameter subgroup of $G$ with $g^1=g$, then $T_g$ is isotopic to the identity map $T_{g^0}=T_e$ where $e$ is the identity of the group $G$. So neither the stable foliation nor the unstable foliation of $T_g$ carries a non-trivial homology.
\end{remark}

\subsection{Concluding remarks}

In \cite{SY}, the authors conjectured the topological entropy is continuous on the space of $C^1$ PHDs with $1$-D center. They obtain the following list of PHDs at which the topological entropy function is continuous in $C^1$ topology (except that item (7) is new).
\begin{enumerate}
  \item uniformly hyperbolic diffeomorphisms,
\item skew products over uniformly hyperbolic diffeomorphisms, with the fiber being a circle,
and perturbations (\cite{BW, Pa, Ha, HP}),
\item derived from Anosov diffeomorphisms (\cite{Ma, HP, Po}),
\item non-dynamically coherent examples of Hertz-Hertz-Ures (\cite{RRU1}),
\item time-one maps of Anosov flows and perturbations (\cite{Hua, HZ, SY}),
\item new examples by Bonatti-Parwani-Potrie (\cite{BPP}) and Bonatti-Gogolev-Potrie (\cite{BGP}),
\item PHDs with uniformly compact center foliation of dimension one (\cite{WWZ}).
  \end{enumerate}

In this paper, we have investigated PHDs with higher dimensional center. Now we can add to the above list the following PHDs at which the topological entropy function $h_{\top}:C^1(M)\to \RR^+\cup\{0\}$ and the unstable topological entropy function $h^u_{\top}:C^1(M)\to \RR^+\cup\{0\}$ are both lower semicontinuous:
\begin{enumerate}[resume]
  \item topologically transitive PHDs with bounded expansion in the center direction (including time-one maps of frame flows; semisimple translations on homogenous spaces; time-one maps of geodesic flows in symmetric spaces of noncompact type, etc.),
\item dynamically coherent PHDs which have subexponential growth in the center direction and satisfy property (EM) (including translations on homogenous spaces),
    \item ergodic toral automorphisms (see also \cite{HSX}).
  \end{enumerate}
We also proved that the topological entropy function $h_{\top}:C^\infty(M)\to \RR^+\cup\{0\}$ is continuous at these PHDs (if we assume PHDs in items (8) and (9) are $C^\infty$). The unstable topological entropy function $h^u_{\top}:C^1(M)\to \RR^+\cup\{0\}$ is always upper semicontinuous. The topological entropy function $h_{\top}:C^1(M)\to \RR^+\cup\{0\}$ is continuous at item (10) and translations on homogenous spaces in item (9).

\ \
\\[-2mm]
\textbf{Acknowledgement.} This work is supported by NSFC Nos. 12071474 and 11701559. The author would like to thank the referee for many helpful suggestions.

\end{document}